\documentclass{amsart}

\usepackage{amsmath, amsthm, amssymb}

\newtheorem{thm}{Theorem}[section]

\newtheorem{lem}[thm]{Lemma}
\newtheorem{cor}[thm]{Corollary}
\newtheorem{prop}[thm]{Proposition}

\theoremstyle{definition}
\newtheorem{dfn}[thm]{Definition}
\newtheorem{prb}[thm]{Problem}
\theoremstyle{remark}

\newcommand{\pmp}{p.m.p.~}
\newcommand{\im}{\operatorname{im}}

\newcommand{\R}{\mathbb{R}}
\newcommand{\N}{\mathbb{N}}
\newcommand{\Z}{\mathbb{Z}}

\newcommand{\baire}{\mathcal{N}}

\newcommand{\inv}{^{-1}}
\newcommand{\ip}[1]{\langle {#1} \rangle} 
\newcommand{\sm}{\smallsetminus}
 
\newcommand{\bp}[2]{\mathbf{\Pi}_{#2}^{#1}} 
\newcommand{\bs}[2]{\mathbf{\Sigma}_{#2}^{#1}} 
\newcommand{\bd}[2]{\mathbf{\Delta}_{#2}^{#1}} 
\newcommand{\essup}{\operatorname{ess\;sup}}

\newcommand{\cost}{\operatorname{cost}}
\newcommand{\dom}{\operatorname{dom}}

\newcommand{\res}{\upharpoonright}
\newcommand{\out}{\operatorname{out}}
\renewcommand{\o}{\mathfrak{o}}
\newcommand{\m}{\operatorname{\mu}}
\newcommand{\ind}{\operatorname{in}}

\newcommand{\lra}{\Leftrightarrow}

\newcommand{\s}[1]{\mathcal{{#1}}} 

\title{Orienting Borel Graphs}
\author{Riley Thornton}
\address{University of California, Los Angeles \\ Department of Mathematics}
\email{personpants@math.ucla.edu}

\begin{document}

\maketitle

\begin{abstract}
    We investigate when a Borel graph admits a (Borel or measurable) orientation with outdegree bounded by $k$ for various cardinals $k$. We show that for a \pmp graph $G$, a measurable orientation can be found when $k$ is larger than the normalized cost of the restriction of $G$ to any positive measure subset. Using an idea of Conley and Tamuz, we can also find Borel orientations of graphs with subexponential growth; however, for every $k$ we also find graphs which admit measurable orientations with outdegree bounded by $k$ but no such Borel orientations. Finally, for special values of $k$ we bound the projective complexity of Borel $k$-orientability for graphs and graphings of equivalence relations. It follows from these bounds that the set of equivalence relations admitting a Borel selector is $\bs12$ in the codes, in stark contrast to the case of smooth relations.
\end{abstract}

\section{Introduction}
\subsection{Basic Definitions and Background}

In this paper we study descriptive set theoretic variants of the following problem: What is the minimum outdegree we need to orient a given graph?

\begin{dfn}
For $G$ a graph on a set $X$, we say that an orientation $o$ of $G$ is a \textbf{$k$-orientation} if $\sup_{x\in X} \out_o(x)\leq k$. The \textbf{orientation number} of $G$, $\o(G)$, is the least cardinal $k$ so that $G$ is $k$-orientable. Equivalently:
    \[\o(G):=\min\big\{\sup_{x\in X}\out_{o}(x): o\mbox{ an orientation of }G \big\}.\]

If $X$ is a standard Borel space and $\mu$ is a Borel measure on $X$, then the Borel and $\mu$-measurable orientation numbers are defined similarly:
   \[\o_B(G):=\min\big\{\sup_{x\in X}\out_{o}(x): o\mbox{ a Borel orientation of }G \big\}\]
    \[\o_\mu(G):=\min\big\{\sup_{x\in A}\out_{o}(x): o\mbox{ a Borel orientation of }(G\res A), \;\m(X\sm A)=0 \big\}.\]

\end{dfn}

For example, $\o(G)=\lceil d/2\rceil$ whenever $G$ is a $d$-regular finite graph, $\o(G)=1$ whenever $G$ is acyclic, and an undirectable forest of lines, as in \cite[Example 6.8]{KM}, is equivalent to a 2-regular Borel graph with $\o_B(G)=2$.

 These notions have already appeared implicitly in the literature on descriptive graph combinatorics. In \cite{Marks}, Marks finds 3-regular graphs with $\o_\mu(G)\leq2$ which admit no sourceless orientations, or equivalently with $\o_B(G)=3$. And the Luzin-Novikov uniformization theorem says that if $o_B(G)\leq \aleph_0$, then it is equal to the number of (not necessarily injective) functions needed to generate $G$. This is closely related to the parameter $k_{\mathcal B}$ studied by Cs\'oka, Lippner, and Pikhurko~\cite{clp}, and we can rephrase a long-standing open problem as follows:

\begin{prb}{\cite[Problem 4.8]{KM16}} If $\o_B(G)$ and $\chi_B(G)$ are finite, does it follow that
\[\chi_B(G)\leq 2\o_B(G)+1?\]
\end{prb}

Further, orientation numbers provide an interesting test case for adapting methods from finite combinatorics. In the classical setting, $\o(G)$ is well understood in terms of partitions of $G$ into sidewalks.\footnote{A sidewalk is graph where each component has at most 1 cycle, i.e.~where bicycles are not allowed. These are sometimes called psuedoforests.}

\begin{prop}
For any graph $G$ on any set $X$,
\[
    \o(G) =\min\big\{|S|: S\subseteq \s P(G),\; \bigcup S=G,\mbox{ and  }(\forall s\in S)\; s\mbox{ is a sidewalk}\big\}
\]
and if $\o(G)$ is finite,
\[\o(G)=\max_{S\subseteq X\mbox{ finite}} \left\lceil\frac{|S|}{\rho(S)}\right\rceil\] where $\rho(S)=|S|-\#\mbox{acyclic components of }G\res S$.
\end{prop}

The first statement is elementary. For the second, observe that the sidewalks in a graph are the independent sets of a matroid (usually called the bicircular matroid). The result then follows from compactness and the Edmonds covering theorem {\cite[Theorem 5.3.2]{matroids}}. Curiously, this characterization in terms of sidewalk covering fails in definable contexts, but we can still recover a measurable version of the Edmonds formula.

\subsection{Statement of Results}

 For a graph $G$, we say a measure $\mu$ is $G$-invariant if $G$ is a countable union of $\mu$-preserving involutions. If $\mu$ is a $G$-invariant probability measure, we say $G$ is probability measure preserving (or p.m.p.). The main result of this paper is the following upper bound for $\o_\mu$ on bounded degree p.m.p. graphs.
 \begin{thm}[See Corollary \ref{thm:nashwilliams}]\label{thm:nashwilliamsintro} If $G$ is a \pmp graph with bounded degree, and $k$ is an integer with $k>\cost(G\res A)$ for $\mu(A)>0$, then $\o_\mu(G)\leq k$.
 \end{thm}
 
In many cases, this upper bound turns out to be optimal, and we also have a sharper analysis for expansive graphs (see Theorem \ref{thm:expansive}).

A measure is $G$-quasi-invariant if it admits a Radon--Nikodym cocycle. That is, a Borel function $\rho:G\rightarrow \R^+$ such that, for any Borel partial injection $f$ contained in $G$ and any Borel $A\subseteq V$, $\mu(f[A])=\int_{x\in A}\rho\left(x,f(x)\right)d\mu$. Details can be found in \cite{KM}. Our proof of Theorem \ref{thm:nashwilliamsintro} is robust in the sense that it gives a bound on $\o_\mu$ for $\mu$ quasi-invariant in terms of the essential supremum of the Radon--Nikodym cocycle (see Theorem \ref{thm:nashwilliams}). As in Conley--Tamuz \cite{ClintonsTrick}, this yields a Borel result for graphs of subexponential growth:

\begin{thm}[See Corollary \ref{cor:borel}]
If $G$ is a $d$-regular Borel graph with subexponential growth, then $\o_B(G)\in [\frac{d}2,\frac{d}2+1]$.
\end{thm}

 In the third section, we apply these results to compute the measurable orientation numbers of Schreier graphs of \pmp group actions in many cases. Using determinacy results, we get a lower bound for the Borel orientation number of a $(\Z/2\Z)^{*n}$ action which is strictly larger than its measurable orientation number (see Theorem \ref{thm:determinacy}). And we also find interesting examples of graphs with uncountable Borel orientation number, including the unit distance graph in $\R^2$ (see Theorem \ref{thm:distancegraphs}).

In the last section, we bound the projective complexity of orientation problems. We show that the set of codes for locally countable Borel $1$-orientable graphs is $\bd12$ using a dichotomy theorem of Hjorth and Miller (see Theorem \ref{thm:1comp}), and the set of codes for Borel $\aleph_0$-orientable graphs is $\bp11$ using a Gandy--Harrington forcing argument (see Corollary \ref{thm:Ncomp}). We also show that the set of closed equivalence relations which admit a $1$-orientable graphing is $\bs12$ complete (see Theorem \ref{thm:Selcomp}). By a theorem of Hjorth, this is equivalent to showing that equivalence relations which admit a selector is $\bs12$ complete, an interesting statement in its own right.

\subsection{Acknowledgements} Thank you to Tyler Arrant, Anton Bernsteyn, Clinton Conley, Alexander Kechris, Andrew Marks, Alexander Mennin, and Oleg Pikhurko for helpful comments and conversations. Thanks to the reviewer for helpful suggestions. This material is based upon work supported  by the National Science Foundation under Grant No.~DMS-1764174.

\subsection{Notation} We end this introduction by settling some notation. We view a graph $G$ on a set $X$ as symmetric subset of $X^2$. To emphasize that $G$ is symmetric, we will often write $\{x,y\}\in G$ if $(x,y)\in G$. For $A\subseteq X$, $(G\res A)$ is the graph on $A$ defined by $(G\res A):=G\cap A\times A.$ An orientation of $G$ is relation $o\subseteq X^2$ so that $(x,y)$ and $(y,x)$ are never both in $o$ and $\{x,y\}\in G$ if and only if $(x,y)\in o$ or $(y,x)\in o$.

A path in $G$ is sequence of vertices $p=(x_0,...,x_n)$ with $\{x_i,x_{i+1}\}\in G$ for $i=0,...,n-1.$ (Some authors refer to such a sequence as a walk and reserve the name path for what we would call a simple path). We say that $p$ is a path from $x_0$ to $x_n$. The length of the path $p$ is $n$, the number of edges it crosses. Given an orientation $o$, We say $p$ is an oriented path if $(x_i,x_{i+1})\in o$ for $i=0,...,n-1$. The radius $n$ positive neighborhood of a vertex is 
\[B^+_{n}(x):=\{y:\mbox{ there is an oriented path from $x$ to $y$ of length at most }n\}\] and the negative neighborhood is \[B^-_{n}(x):=\{y:\mbox{ there is an oriented path from $y$ to $x$ of length at most }n\}.\] And, for sets $A\subseteq X$ we define
\[B^+_n(A)=\bigcup_{x\in A} B^+_n(x),\quad \mbox{and } B^-_n(A)=\bigcup_{x\in A} B^-_n(x).\] We abbreviate $B^+_1(A)=B^+(A).$

For any relation $R$, the section of $R$ at $x$ is $R_x=\{y: (x,y)\in R\}$. For any vertex $x$, the degree of $x$ is $\deg(x)=|G_x|$. The outdegree of $x$ with respect to $o$ is $\out_o(x)=|o_x|$. The indegree of $x$ is $\ind_o(x)=\deg(x)-\out_o(x)$. If $G$ is a graph on a measure space $(X,\mu)$, then the cost of $G$ is half of the average degree of $G$:
\[\operatorname{cost}(G)=\frac{1}{2}\int_{x\in X} \deg(x)\;d\mu.\]
For $\mu$ $G$-quasi-invariant, we abuse notation and use $\mu$ also for the associated measure on $G$: for $S\subseteq G$
\[\m(S):=\int_{x\in X} |S_x|\;d\mu.\]
The expansion constant of a graph is
\[\lambda_\mu(G)=\inf\left\{\frac{1}{\m(A)} \int_{x\in A} |G_x\sm A|\;d\mu: 0<\mu(A)\leq \frac{1}{2}\right\}.\] The quantity minimized by $\lambda_\mu$ is the measure theoretic analog of the ratio of edges leaving $A$ to vertices in $A$. We call a graph expansive if $\lambda_\mu(G)>0$.

We will sometimes suppress $\mu$ and $o$ in notation. We will also conflate $o$ with its characteristic function to speak of $o(e)$ and $\lim_i o_i(e)$ for sequences of orientations.


\section{Bounds via Measurable Combinatorics}

For this section, $G$ will be a locally countable Borel graph, the letter $\mu$ will always stand for a $G$-quasi-invariant Borel probability measure on $X$, $\rho(x,y)$ stands for the associated Radon--Nikodym cocycle, and $\rho:=\essup_{x,y} \rho(x,y)$. In particular $\mu$ is $G$-invariant if and only if $\rho=1$. We include a proof of the following basic proposition to give a flavor of how $\rho$ appears in arguments.

\begin{prop}
For any $\mu$, orientation $o$, 
\[\int_{x\in X} \ind_o(x)\;d\mu \leq \rho\int_{x\in X} \out_o(x)\;d\mu.\]
\end{prop}
\begin{proof}
Fix a set of involutions generating $o$, $\{f_i: i\in\N\}$. We may choose these functions so that, for any $x\in X$,
\[\out_o(x)=\sum_{i\in\N} \mathbf{1}_{\dom(f_i)}(x) \quad \mbox{and }\ind_o(x)=\sum_{i\in\N} \mathbf{1}_{\im(f_i)}(x).\] Then since each $f_i$ is injective,
\begin{align*}
    \int_{x\in A} \ind_o(x) \;d\mu &= \sum_{i\in\N} \m(\im(f_i)) \\
    \; & =\sum_{i\in\N} \int_{x\in \dom(f_i)} \rho(x,f_i(x)) \;d\mu \\
    \; & \leq \int_{x\in A} \rho \sum_{i\in\N} \mathbf{1}_{\dom(f_i)}(x) \;d\mu \\
    \; & = \rho \int_{x\in A} \out_o(x) \;d\mu.
\end{align*}
\end{proof}
We use similar propositions throughout this section without comment. Readers unfamiliar with quasi-invariant measures can consult \cite{KM}, or consider the \pmp case where $\rho=1$ to get a gist of the arguments.

\begin{dfn} 
Let $\alpha_\mu(G)=\sup\{\cost(G\res A): A\subseteq X, \m(A)>0 \},$ where the cost of a restriction is computed with respect to the normalized measure $\mu/\mu(A).$ That is,
\[\cost(G\res A)=\frac{1}{2\mu(A)}\int_{x\in A} |G_x\cap A| \;d\mu.\]
\end{dfn}

For example, if $G$ is $d$-regular, $\alpha_\mu(G)=\frac{d}2.$ And, if $G$ is a finite graph equipped with counting measure, this is essentially the Edmonds formula for $\o$.

\begin{prop}\label{prop:cost} For any $\mu$,
\[\o_\mu(G)\geq \left\lceil\frac{2\alpha_\mu(G)}{1+\rho}\right\rceil.\]
In particular, if $\mu$ is $G$-invariant,
$\o_\mu(G)\geq \left\lceil \alpha_\mu(G)\right\rceil$.
\end{prop}
\begin{proof}
Since $\o_\mu(G\res A)\leq \o_\mu(G)$ when $\m(A)>0$, it suffices to show \[\o_\mu(G)\geq \frac{2}{1+\rho}\cost(G).\]

Suppose $G$ has an orientation $o$ with outdegree bounded by $n$. Then,
\begin{align*}\int_X |G_x|\; d\mu &  \leq \int_X \out_o(x)\; d\mu+\int_X \ind_o(x) \;d\mu. \\
\; & \leq (1+\rho) \int_X \out_o(X)\; d\mu \end{align*}
so
\[\operatorname{cost}(G)=\frac{1}{2}\int_X |G_x|\; d\mu \leq \frac{1+\rho}{2}\int_X \out_o(x)\; d\mu\leq \frac{1+\rho}{2}n.\]
\end{proof}

The next theorem says that, for bounded degree graphs with $\rho$ small this lower bound is close to sharp.

\begin{thm}\label{thm:nashwilliams} If $G$ has bounded degree, then for any $k\in\N$, there is a Borel orientation $o$ such that, for any measure $\mu$, if $k>\rho^2\alpha_\mu(G)$, then
\[\mu(\{x: \out_o(x)>k\})=0.\]

\end{thm}

\begin{proof}
The proof is similar to the proof of the Lyons--Nazarov matching theorem in Elek--Lippner \cite{Matching}, but uses a different notion of augmenting chain. 

Let $\Delta$ be a degree bound for G. Given an orientation $o$ of $G$, Define \[O_o:=\{x: \out_o(x)> k\}, \quad \mbox{and } I_o:=\{x:\out_o(x)< k\}.\] We say that a path $(x_0, x_1,...,x_n)$ is an augmenting chain in $o$ if it is an oriented path from $O_o$ to $I_o$, i.e.~if
\begin{enumerate}
    \item $(x_i,x_{i+1})\in o$ for $1\leq i\leq n$, and
    \item $\out_o(x_1)> k$ and $\out_o(x_n)< k$.
\end{enumerate} We say that $o'$ is gotten by flipping an edge $(x,y)$ in $o$ if \[o'(e)=\left\{\begin{array}{ll} 1-o(e) & e=(x,y) \mbox{ or }(y,x)\\ o(e) & \mbox{else} \end{array} \right.\] The key observation is that flipping every edge in an augmenting chain only changes the outdegrees of the endpoints of the chain.

\begin{lem}\label{lem:augment} For any Borel orientation $o$, there is a Borel orientation $o'$ such that 
\begin{enumerate}
\item $o'$ does not admit any augmenting chains of length less than $n$, and 
\item  For any $\mu$, \[\m(\{e:o(e)\not=o'(e)\})\leq 2n \Delta \rho^n \min(\m(O_o),\m(I_o)).\]
\end{enumerate}
\end{lem} 
The $\m(I_o)$ bound will be used in a later theorem.
\begin{proof}[Proof of lemma] Let $\mathcal{C}\subseteq X^{\leq n}$ be the space of (unoriented) paths in $G$ of length at most $n$. By a theorem of Kechris and Miller \cite[Proposition 3.10]{KM16}, there is a countable coloring of the intersection graph on $\mathcal{C}$. Let $\ip{\ell_i:i\in\N}$ enumerate $\N$ with each number repeated infinitely often. Define $o_i$ inductively:
\begin{enumerate}
    \item $o_0=o$
    \item Get $o_{i+1}$ by flipping all paths in the $\ell_i^{th}$ color class which are augmenting chains for $o_{i}$.
\end{enumerate}

 First we check that this process converges on every edge, then we verify that the limiting orientation is as desired.
 
 For any vertex $x$, $\out_{o_i}(x)$ is monotone (this is where we use that $k$ is an integer). And $\out_{o_i}(x)$ drops if and only if $x$ is the starting vertex of a chain flipped at stage $i$. So a given chain can flip at most $\Delta$ many times in this process. Since each edge is contained in only finitely many chains, each edge only flips finitely often. We can then define $o'(e):=\lim_i o_i(e)$.

To see that $(1)$ of the lemma holds, note that if $p$ were an augmenting chain $o'$, it would be an augmenting chain in cofinally many $o_i$. This is because $o_i(e)$ stabilizes in finite time for all $e$. But then $p$ would have been flipped at cofinally many stages, which is absurd. Thus $o'$ does not admit any augmenting chains of length smaller than $n$.

For $(2)$ of lemma, if $o(e)\not=o'(e)$, we can define $f(e)$ to be the first vertex of the first flipped chain containing $e$ (before or after the flip). If $f(x,y)=v$ and $e_1,...,e_d$ are the edges incident to $v$, then one of $(x,y)$ or $(y,x)$ is on the first chain flipped that contains one of the $e_i$'s. So $f$ is at most $2n\Delta$-to-one. Further, $f\inv(x)$ is empty unless $x\in O_o$, and there are at most $n$ edges between $x$ and $f(x,y)$. We compute

\begin{align*} \m(\{e: o(e)\not=o'(e)\}) & = \int_{x\in X} | \{y: o(x,y)\not=o'(x,y)\}|\;d\mu \\
                    \; & \leq \int_{x\in X} \rho^n|f\inv(x)| \;d\mu\\
                    \; & \leq 2n\Delta \rho^n \m(O_o).
\end{align*} Similarly, setting $f(e)$ to be the final vertex of the first flipped chain containing $e$ gives $\m(\{e: o(e)\not=o'(e)\})\leq 2n\Delta\rho^n\m(I_o).$
\end{proof}

Now we want to iterate the construction above to get an orientation with no augmenting chains. To ensure this converges $\mu$-a.e.~for any appropriate $\mu$, we need to analyze $\m(O_o)$. To this end, fix $\mu$ and abbreviate $\alpha=\alpha_\mu(G)$. Recall
\[B^+_n(A):=\{y: \mbox{ there is an oriented path of length at most $n$ from some $x\in A$ to }y\}\] and $B^+(A):=B^+_1(A)$. Notice that $B^+_{n+m}(A)=B^+_n(B^+_m(A))$. 

\vspace{11pt} 

 \emph{Claim 1:} If every point in $A\subseteq G$ has outdegree at least $k$, then 
\[\m\left(B^+(A)\right)\geq \frac k {\rho\alpha} \m(A).\]

Using the facts that every edge coming out of A ends in $B^+(A)$ and that ${\ind(x)+\out(x)}=\deg(x)$ for the graph $G\res B^+(A)$, we have
\begin{align*}k\m(A) & \leq \int_{x\in A} \out_o(x)\; d\mu \\ 
\; & \leq \frac12\left(\int_{x\in A} \out_{o\res B^+(A)}(x) \; d\mu +\rho\int_{x\in B^+(A)} \ind_{o\res B^+(A)}(x)\;d\mu\right)\\
\; & \leq \frac{\rho}{2} \int_{x\in B^+(A)}\left( \out_{o\res B^+(A)}(x)+\ind_{o\res B^+(A)}(x)\right)\;d\mu  \\
\; & \leq \frac\rho2 \int_{x\in B^+(A)} \left|G_x\cap B^+(A)\right|\;d\mu \\
\; & =\rho\cost \left(G\res B^+(A)\right) \m\left(B^+(A)\right) \\ 
\; & \leq \rho \alpha \m\left(B^+(A)\right) .\end{align*} 

It follows from the above claim that $\mu(O_o)$ shrinks exponentially in the length of the smallest augmenting chain in $o$.

\vspace{11pt} 

\emph{Claim 2:} If $o$ admits no augmenting chains of length at most $n$, then $O_o$ has measure at most $\left(\frac {\rho\alpha} {k} \right)^n$.

In this case, any oriented path starting in $O_0$ of length at most $n$ must fail to reach a vertex of outdegree less than $k$. That is, for $i\leq n$, $B^+_i(O_o)$ satisfies the hypotheses of Claim $1$. So,
\[1\geq \m\left(B^+_n(O_0)\right)\geq \left(\frac{k}{\rho\alpha}\right)^n \m(O_0).\]

Now to get the orientation $o$ of the theorem statement, start with any Borel orientation $o_0$ of $G$ and iteratively produce $\ip{o_i: i\in\N}$ with $o_{i+1}=o_i'$ as in Lemma \ref{lem:augment}. Let $o(e)=\lim_i o_i(e)$ if the limit exists, otherwise set $o(e)=o_0(e)$. Note that this construction does not depend on $\mu$. The probability that an edge flips at stage $n$ is at most $2n\Delta\m(O_{o_n})\leq 2n\Delta \frac{(\rho^2\alpha)^n}{k^n}$, which is summable. By Borel-Cantelli $o_i(e)$ converges almost surely. Thus, away from a $\mu$-null set $o$ admits no augmenting chains, and by Claim 2 $O_o$ is $\mu$-null.

\end{proof}
For \pmp graphs, we get the following.
\begin{cor} \label{cor:nashwilliams}
If $G$ is \pmp with measure $\mu$, then \[\o_\mu(G)\in[\alpha_\mu(G),\alpha_\mu(G)+1].\] In particular, if $\alpha_\mu(G)$ is not an integer, $\o_\mu(G)=\lceil\alpha_\mu(G)\rceil.$
\end{cor}
It follows from this that any $d$-regular p.m.p.~graph can be generated by $\lceil (d+1)/2 \rceil$ functions. Similar results have been obtained by Greb\'ik and Pikhurko \cite[Theorem 1.7]{pg}. Using an idea from Conley--Tamuz \cite{ClintonsTrick}, we can get a Borel result for slow-growing graphs.
\begin{cor}\label{cor:borel}
If $G$ has subexponential growth and is $d$-regular, then
$\o_B(G)\in[\frac{d}{2},\frac{d}2+1].$
\end{cor}
\begin{proof}
In this case, for any vertex $x$ of $G$, there is an atomic measure with $\rho<\sqrt{1+d}$ whose support contains $x$. The Borel orientation given by Theorem \ref{thm:nashwilliams} witnesses $\o_\mu(G)\leq \frac{d}2+1$ for all of these atomic measures, so witnesses $\o_B(G)\leq \frac{d}{2}+1$. Similarly, Proposition \ref{prop:cost} gives a lower bound.
\end{proof}

We can also sharpen the analysis for expansive regular graphs. If $G$ is expansive, then edge boundaries in $G$ are large. So $\cost(G\res A)$ is bounded away from $\alpha_\mu(G)$ when $A$ has measure less than $1/2$. If $G$ is regular, the problem is symmetric enough that we only need to consider these small sets.

\begin{thm}\label{thm:expansive}
If $\mu$ is $G$-invariant, and $G$ is $d$-regular and expansive
then \[\o_\mu(G)=\left\lceil \frac{d}{2}\right\rceil.\]
\end{thm}
\begin{proof}
It is enough to consider the even case. We modify the proof of Theorem \ref{thm:nashwilliams}, making use of the following symmetry. For any orientation $o$, let $o\inv=\{(x,y): (y,x)\in o\}$. That is, $o\inv$ is $o$ with every edge flipped. We have $\out_o(x)=\ind_{o\inv}(x)$, $B^+_o(A)=B^-_{o\inv}(A)$, and
\[I_{o}=\{x: \out_o(x)<d/2\}=\{x: \ind_o(x)>d/2\}=O_{o\inv}.\] 
\vspace{11pt}
\emph{Claim 1:} There is $c>1$ such that, for any orientation $o$, if $A\subseteq X$ satisfies ${0<\m\left(B^+(A)\right)\leq 1/2}$ and $\out_o(x)\geq d/2$ for all $x\in A$, then
\[c \m(A)\leq\m\left(B^+(A)\right) \] 

Let $\lambda$ be the expansion constant of $G$. Then
\begin{align*} \frac{d}{2}\m(A) & \leq\int_{x\in A}\out_o(x) \; d\mu\\ 
\; & \leq \frac{1}{2}\int_{x\in B^+(A)} \left(\deg(x)-|G_x\sm B^+(A)| \right) d\mu \\ 
\; & \leq \frac{1}{2}(d-\lambda)\m\left(B^+(A)\right). \end{align*}So, $c=\frac{d}{d-\lambda}$ works.

Symmetrically, if $\m\left(B^+(A)\right)\leq \frac{1}{2}$ and $\out_o(A)\leq d/2$ for $x\in A$, then applying Claim $1$ to $o\inv$ gives
\[c\m(A)\leq \m\left(B^-(A)\right).\] 

As in the proof of \ref{thm:nashwilliams}, we derive an exponential bound on $\min(\m(O_o),\m(I_o)).$
\vspace{11pt}

\emph{Claim 2:} There is $c>1$ such that, if $o$ admits no augmenting chains shorter than length $2n$, then  $\min\left(\m(I_o),\m(O_o)\right)< c^{-n}$.

In such an orientation, $B^+_n(O_o)\cap B^-_n(I_o)=\emptyset$. So one of them must have measure bounded by $1/2.$ Possibly replacing $o$ with $o\inv$, we can assume $\m\left(B^+_n(O_o)\right)\leq 1/2$. By claim 1,
\[\frac{1}{2}\geq \m\left(B^+_{n}(O_o)\right)\geq c^{n}\m(O_o) \]
for some $c>1.$ Thus, $\m(O_o)\leq \frac{1}{2c^n}$

Now again, iteratively produce $o_i$ with $o_{i+1}=o_i'$ as in Lemma \ref{lem:augment}. Claim 2 and Borel-Cantelli imply that the probability of an edge flipping infinitely often is 0. Let $o$ be the limiting orientation.

Since $o$ admits no augmenting chains, claim 2 implies one of $O_o$ or $I_o$ is null. If it is $O_o$, $o$ is a $\frac{d}{2}$-orientation. Otherwise $o\inv$ is a $\frac{d}{2}$-orientation.
\end{proof}

\section{Examples}

The most widely studied class of Borel graphs are the Schreier graphs associated to actions of finitely generated groups. Since we are dealing with orientations, we do not want generating sets for our groups to be symmetric. And to avoid some degenerate cases, we will also not allow the identity to be in our generating sets.

\begin{dfn}
For a countable group $\Gamma$, we say $E$ generates $\Gamma$ if \[\Gamma=\{a_1....a_n: a_i\in E\cup E\inv\}\] and $e\not\in E$. Importantly, we are allowed to take inverses of our generators.

If $\Gamma$ is generated by $E$ and $a:\Gamma\curvearrowright X$ is a Borel group action, then the associated Schreier graph on $X$ is
\[G(a, E):=\{\{x,y\}: (\exists g\in E)\; g\cdot x=y\}.\]

The graph $F(\Gamma, E)$ is the graph $G(a, E)$ where $a$ is the shift action of $\Gamma$ on the free part of $\N^\Gamma.$
\end{dfn}

For any free action $a$, $G(a, E)$ is regular with degree $|E\cup E\inv|$. In particular, if $E$ contains no involutions, $G(a, E)$ is $2|E|$-regular. So, if $a$ is a free \pmp action, then $\o_\mu\left(G(a, E)\right)\geq \frac{1}{2} |E\cup E\inv|$ by Proposition \ref{prop:cost}. Clearly, $\o_B\left(G(a, E)\right)\leq |E|$. We then have

\begin{prop}
If $E$ contains no involutions, and $a$ is a p.m.p action of $G$, then $\o_\mu\left(G(a,E)\right)=\o_B\left(G(a,E)\right)=|E|.$
\end{prop}

The situation is much more interesting for groups with $2$-torsion. We equip $\N^\Gamma$ with any non-atomic product measure $\mu$.

\begin{prop}\label{prop:z2}
Let $\Gamma=(\Z/2\Z)^{*2}:=\ip{a,b: a^2=b^2=1}$ and $E=\{a,b\}$. Then,
\[\o_\mu\left(F(\Gamma, E)\right)=2.\]
\end{prop}
In particular, there are \pmp graphs with $\o_\mu(G)\not=\lceil\alpha_\mu(G)\rceil.$
\begin{proof}
Each component of $F(\Gamma,E)$ is an infinite path. Suppose toward contradiction that $o$ is a measurable $1$-orientation, so $o$ assigns a direction to each component of $F(\Gamma, E)$. Consider $A=\{x: (x,ax)\in o\}.$ The group element $ab$ acts ergodically and $A$ is invariant under this action. Thus almost every point is in $A$ or almost no point is in $A$. In either case, some edge must be missed by $o$, which is a contradiction.
\end{proof}

Conley has shown that $F(\Gamma, \{a,b,ab\})$ is a 4-regular graph with no measurable 2-orientation \cite{private}. It is natural to ask the following

\begin{prb}
Is there a $2d$-regular graph with no measurable $d$-orientation?
\end{prb}

Since submission of this article, Bencs, Hru\v skov\'a, and T\'oth have answered this question in the affirmative \cite{laszlo}.

Shift graphs of nonamenable groups are expansive (see e.g. {\cite[Section 3]{LN}}). So by Theorem \ref{thm:expansive}, $\o_\mu\left(F(\Gamma, E)\right)=\lceil \frac n2\rceil$ when $\Gamma=(\Z/2\Z)^{*n}$ with standard generating set $E$. Using a determinacy result, we can show that these graphs have strictly larger Borel orientation numbers.

\begin{thm}\label{thm:determinacy}
For $\Gamma=(\Z/2\Z)^{*n}=\ip{a_1,...,a_n: a_i^2=1}$, $E=\{a_1,...,a_n\}$,  \[\o_B\left(F(\Gamma, E)\right)=n\]
\end{thm}
\begin{proof} It suffices to consider the case when $n$ is even. We proceed by induction. The base case is Proposition \ref{prop:z2}.

For the induction step, we use the main lemma from Marks's paper on Borel determinacy \cite[Lemma 2.1]{Marks}. Suppose the shift graph for $(\Z/2\Z)^{*(n-2)}$ does not admit an $(n-3)$-orientation, and suppose $o$ is an orientation of $F(\Gamma, E)$ with $\out_o(x)<n$ for all $x$. Set $H=\ip{a_1, a_2}$ and $K=\ip{a_3,...,a_n}$, and let \[A=\{x: (x,a_1x)\not\in o\mbox{ or }(x,a_2x)\not\in o\}.\] By the lemma,  there is an equivariant embedding either of $F(H, \{a_1,a_2\})$ into $A$ or of ${F(K, E\sm\{a_1,a_2\})}$ into $X\sm A$. 

In the first case, if $f$ is the embedding, let $\tilde o$ be the pullback orientation on $F(H,E)$, i.e. for $i=1,2$
\[(x,a_ix)\in\tilde o\lra (f(x),a_if(x))\in o.\] By the definition of $A$, $\tilde o$ is a 1-orientation of $F(H,E)$, contradicting the base case.

In the second case, again suppose $f$ is the embedding and let $\tilde o$  be the pullback orientation of ${F(\Gamma, E\sm \{a_1,a_2\})}$. Then for all $x$, 
\[(f(x),a_1 f(x)), (f(x),a_2 f(x))\in o\] so
\[|\tilde o_x|=\big|\{a_i: ( f(x), a_i f(x))\in o , i\not=1,2\}\big|= |o_{f(x)}|-2<n-2. \] But then we have an $n-3$ orientation of $F(K, E\sm\{a_1,a_2\})$, which contradicts the induction hypothesis.

\end{proof}

Another nice class of graphs generalizes the Hadwiger--Nelson graph.

\begin{dfn}
For a Polish group $\gimel$ and $E\subseteq \gimel$ Borel, the associated generalized distance graph is
\[D(\gimel,E):=\left\{\{x,y\}: xy\inv\in E\right\}.\]
\end{dfn}

When $E$ is countable, this is just the Schreier graph of $\ip{E}$ acting on $\gimel$ by translation. It turns out, when $E$ is uncountable, $\o_B\left(D(\gimel,E)\right)=|\R|$.

\begin{thm}\label{thm:distancegraphs}
For any Polish group $\gimel$ and $E\subseteq \gimel$ Borel, $\o_B\left(D(\gimel, E)\right)\leq \aleph_0$ if and only if $E$ is countable.
\end{thm}
\begin{proof}
Set $G=D(\gimel, E)$. If $E$ is countable, then $G$ is locally countable, and $\o_B(G)\leq \aleph_0$. So, suppose $E$ is uncountable and fix a sequence of Borel functions $\ip{f_i: \in\N}$ with $f_i\subseteq G$ for each $i$. We will show that $\bigcup_i f_i$ is not an orientation of $ G$. We may assume $E$ is symmetric, and by the perfect set theorem we may replace $E$ with one of its uncountable closed perfect subsets.

Equipped with the subspace topology, $ G\subseteq \gimel^2$ is homeomorphic to $\gimel\times E$. There are two natural identifications; a point $(x,e)\in \gimel\times E$ can map to either $(x,ex)$ or $(ex,x)$. We can translate between the two via a self-homeomorphism of $\gimel\times E$, $(x,e)\mapsto (ex,e\inv).$

Define $\tilde f_i:\gimel\rightarrow E$ by $\tilde f_i(g)=f_i(g)g\inv$. Then $\tilde f_i$ is Borel and so has a meager graph in $\gimel\times E$. That is, \[\{(x,e)\in\gimel\times E: (x,e)\not=(x,\tilde f_i(x))\}=\{(x,e): (x,ex)\not=(x,f_i(x))\}\] is comeager. Symmetrically, $\{(x,e): (x,ex)\not=(f_i(ex),ex)\}=\{(x,e): (ex,e\inv)\not\in \tilde f_i\}$ is comeager.  So, for a comeager set of $(x,e)$, then $(x,ex), (ex,x)\not\in \bigcup_i f_i.$
\end{proof}
For $G=\R^n$ and $E=S^{n-1}$, we get the following.
\begin{cor}
The unit distance graph in $\R^n$ ($n\geq 2$) does not have a countable Borel orientation.
\end{cor}
This differs from the classical case. In ZFC, the unit distance graph on $\R^n$ always admits a countable orientation, independent of the size of the continuum. See, for example, \cite[Theorem 6.2]{AZ}.

We end this section by noting that the orientation and sidewalk covering numbers can be arbitrarily far apart in the Borel setting.

\begin{thm}
For every $n\in \N\cup\{\aleph_0, |\R|\}$, there is an acyclic Borel graph (in particular a sidewalk) with $\o_B(G)=n$. Further, if $n\leq \aleph_0$, $G$ can be taken to be locally countable.
\end{thm}

\begin{proof}
For $n\leq\aleph_0$, $F_n$, the free group with $n$ generators, is torsion free. So if $E$ is the usual set of generators, $\o_B(F(F_n, E))=n$. And the Schreier graphs of $F_n$ actions are locally countable.

For $n=|\R|$, label the standard generators for $F_{2^n}$ as $E_n=\{a_\sigma: \sigma\in 2^n\}$ and let $g_n:F_{2^n}\rightarrow F_{2^{n-1}}$ be the homomorphism determined by $g_n(a_\sigma)=a_{\sigma\res (n-1)}$. Define

\[\Gamma=\lim_{\leftarrow} F_{2^n}=\{f\in \prod_n F_{2^n}:(\forall n)\; g_n(f(n))=f(n-1)\}\] \[E=\{f\in \Gamma: (\forall i)\; f(i)\in E_i\}.\] Then $E$ is closed in $\Gamma$ and uncountable, so by Theorem \ref{thm:distancegraphs}, $\o_B(G(\Gamma, E))>\aleph_0$. Also, since every finite subset of $E$ freely generates a free group, $D(\Gamma, E)$ is acyclic.
\end{proof}

\section{Complexity}

 We end with some metamathematical considerations. Since Borel sets admit $\bp11$ coding, we can consider the set of codes for Borel sets with various combinatorial properties. If, like Borel $k$-orientability or colorability for graphs, the property asks for some kind of Borel witness, the set of codes will usually be $\bs12$. Dichotomy results and effective witnesses give better upper bounds on complexity, and consequently lower bounds on complexity can be construed as anti-dichotomy or impossibility results. See \cite{TV} for more discussion.

We can give a strong complexity bound for countable orientability using effective methods. We use an alternate characterization for $\o_B(G)$ countable
\[\o_B(G)=\min\{|F|: F\mbox{ is a family of Borel functions generating }G\}.\] 
Likewise, define
\[\o_{\Delta^1_1}(G):=\min\{|F|: F\mbox{ is a uniformly $\Delta^1_1$ family functions generating }G\}.\] This may give strange values for graphs which are not countably orientable, but it will not matter for our discussion.

\begin{thm} 
If $G$ is $\Delta^1_1$, and $\o_B(G)\leq \aleph_0$, then $\o_{\Delta^1_1}(G)\leq \aleph_0$.
\end{thm}
\begin{proof}

We want to consider generic edges in $G$. Let $\mathbf{P}_n$ be Gandy-Harrington forcing on $\baire^n$, i.e. forcing with nonempty $\Sigma^1_1$ subsets of $\baire^n$. Further, let $(\dot x, \dot y)$ be the canonical name for an $\mathbf{P}_2$-generic pair of reals. Recall that if $(x,y)$ is $\mathbf{P}_2$-generic, then $x$ and $y$ are separately $\mathbf{P}_1$-generic.

Suppose towards contradiction that 
\[A:=G\sm\bigcup\{D\in \Delta^1_1: D\mbox{ is the graph of a partial function}\}\not=\emptyset\] and $G$ is generated by $\{f_i:i\in\N\}$, with $f_i$ Borel. Note that, by the first reflection theorem \cite[Lemma 1.2]{HMS}, $A=G\sm \bigcup\{p\in \Sigma^1_1: p\mbox{ is the graph of a partial function}\}.$

Then, without loss of generality, we can find some $i$ and some Gandy-Harrington condition $p\leq A$ such that \[p\Vdash f_i(\dot x)=\dot y.\] Since $p$ is a nonempty $\Sigma^1_1$ subset of $A$, $p$ is not a partial function. So, we have a nonempty $\mathbf{P}_1$ condition:
\[U''=\{x: (\exists y,y')\; y\not=y' \mbox{ and }(x,y), (x,y')\in p\}.\]
We can refine $U''$ to freeze the first place $y,y'$ differ, i.e. for some $n$ and $U'\subseteq U''$ we get
\[U'\Vdash (\exists y,y')\; (y\res n)\not=(y'\res n)\mbox{ and }(\dot x,y),(\dot x,y')\in p.\] And we can further refine $U'$ to freeze the first $n$ digits of $f_i(x)$. There are $\sigma\in \N^n$ and $U\subseteq U'$ so that
\[U\Vdash f_i(x)\in N_\sigma,\;\left[ (\exists y,y')\; (y\res n)\not=(y'\res n)\mbox{ and }(x,y),(x,y')\in p\right].\]

Now set $q:=U\times (\baire\sm N_\sigma)\cap p.$ By construction, $q$ is nonempty. And if $(x,y)$ is $\mathbf{P}_2$-generic below $q$, then $x$ is $\mathbf{P}_1$-generic below $U$, so $f_i(x)\in N_\sigma$. Similarly, $y\in \baire\sm N_\sigma$. But then
\[q\Vdash f_i(\dot x)\in N_\sigma, f_i(\dot x)=\dot y, \mbox{ and } \dot y\not \in N_\sigma,\] which is a contradiction.
\end{proof}

\begin{cor} \label{thm:Ncomp}
The set of $\aleph_0$-orientable Borel graphs is $\bp11$-complete in the codes.
\end{cor}
\begin{proof}
By the preceding theorem, $c$ codes a graph with $\o_B(G)\leq \aleph_0$ if and only if
\[(\exists f\in \Delta^1_1(c)) \; f\mbox{ codes a countable generating family for }G.\] Coding a countable generating family is $\bp11$, and an existential quantifier over $\Delta^1_1(c)$ is equivalent to a countable quantifier and a universal quantifier. So being $\aleph_0$-orientable is $\bp11.$

We can reduce well-foundedness to $\aleph_0$-orientability as follows. If $E\subseteq \baire$ is closed, then the complete graph on $E\times \baire$ is countably orientable if and only if $E$ is empty. The map $E\mapsto K_{E\times\baire}$ can be carried out in a Borel way on the level of trees.
\end{proof}

Unfortunately, the above proof does not yield a satisfying dichotomy.

\begin{thm} \label{thm:1comp}
The set of locally countable graphs with $\o_B(G)\leq 1$ is $\bd12$ in the codes.
\end{thm}
\begin{proof}
Say that a graph admits Borel end selection if there is a Borel function $f:X\rightarrow X^\N$ so that 
\begin{enumerate}
    \item $f(x)$ is either injective (as a sequence), or constant
    \item If $f(x)$ is injective, then $f(x)(0)=x$ and $\{f(x)(n), f(x)(n+1)\}\in G$ for all $n$.
    \item If $x,y$ are connected in $G$ then, for any finite $S\subseteq X$,  $f(x)$ and $f(y)$ are eventually in the same component of $X\sm S.$ 
\end{enumerate} In particular if $f(x)$ is constant, so is $f(y)$, and $f(x)=f(y)$.

We first show that $\o_B(G)\leq 1$ if and only if $G$ is a sidewalk and admits Borel end selection.

If $G$ is generated by a single function $g$, then $\ip{g^n(x):n\in\N}$ is either injective or is eventually periodic. We can then define a Borel end selector by $f(x)(n)=g^n(x)$ if this is injective or $f(x)(n)$ is the least value $g^n(x)$ repeats (according to some Borel linear order on $X$) otherwise.  

And if $G$ is a sidewalk and admits an end selector $f$, we can generate $G$ with a single function $g$ as follows. Let $A$ be the set of points which are connected to a cycle or which have $f(x)$ constant. Then $G\res A$ has a smooth connectedness relation and it is straightforward to find a function generating the edge set. On $X\sm A$, $f(x)$ is always injective, and $X\sm A$ is acyclic. So if $\{x,y\}\subseteq X\sm A$ is an edge, $f(x)(1)=y$ or $f(y)(1)=x$ (otherwise $f(x), f(y)$ are in different $X\sm \{x\}$ components). So $x\mapsto f(x)(1)$ generates $G\res(X\sm A).$

By a dichotomy theorem of Hjorth and Miller \cite{HjorthMiller}, the set of codes for locally countable graphs admitting a Borel end selector is $\bd12$. The result follows as the set of codes for locally countable sidewalks is $\bd12$.
\end{proof}

Unfortunately, Hjorth and Miller's methods do not seem to translate to an effective proof. Somewhat unusually, then, we have two upper bounds on complexity without all the attending niceness theorems.

\begin{prb}
Find a dichotomy characterizing countable orientability.
\end{prb}
\begin{prb} \label{o1effective}
Is it the case that if $G$ is locally countable, $\Delta_1^1$, and $\o_B(G)\leq 1$, then $\o_{\Delta^1_1}(G)\leq 1$?
\end{prb}

Since submission, these problems have been resolved by the author \cite{effectivization}. The answer to problem \ref{o1effective} is yes, even if we remove the local countability assumption.

Hyperfinite Borel equivalence relation are equivalent to countable Borel relations admitting a  1-orientable graphing. Determining the complexity of such relations is a longstanding open problem. That problem is still out of reach, but we can settle the complexity of general relations admitting a 1-orientable graphing. The following proposition says it is enough to give a lower bound on the complexity of equivalence relations admitting a Borel selector.

\begin{prop}
For $E$ smooth (but not necessarily countable), the following are equivalent:
\begin{enumerate}
    \item $E$ admits a $1$-orientable graphing
    \item $E$ is treeable
    \item $E$ admits a Borel selector
\end{enumerate}
\end{prop}
\begin{proof}
$(2)$ and $(3)$ are equivalent by a result of Hjorth \cite{HjorthTreeable}. To see that $(3)$ implies $(2)$, note that if $f$ is a Borel selector for $E$, then $f$ generates a graphing of $E$.  

Now we show $(3)$ implies $(1)$. Suppose that $\o_B(E)\leq 1$ and fix an graphing $G$ witnessing this. Since $G$ is a sidewalk, each component is either a tree or contains a unique cycle. Tossing out the least edge in each cycle (relative to some Borel linear ordering of the underlying space) yields a treeing of $E$.
\end{proof}

For $X$ Polish, let $F(X)$ be the Effros Borel space of closed subsets of $X$. 

\begin{thm} \label{thm:Selcomp}
The set 
\[\mathbf{Sel}:=\{E\in F(\baire^2):E\mbox{ is an equivalence relation with a Borel selector}\}\] is $\bs12$ complete. 
\end{thm}

\begin{proof}

We prove this in two steps. Define
\[\mathbf{Uni}:=\{R\in F(\baire^2): E \mbox{ admits a Borel uniformization}\}.\] We will show
\[\mathbf{FBU}\leq_B\mathbf{Uni}\leq_B\mathbf{Sel}\] where $\mathbf{FBU}$ is the set of relations with full domain admitting Borel uniformization. Then, by a theorem of Adams and Kechris \cite{AK}, $\mathbf{Sel}$ is $\bs12$ complete.

\paragraph{$\mathbf{FBU}\leq_B\mathbf{Uni}$:} We want to take a relation $R$ and extend it to a relation $R'$ with full domain in such a way that the $R'$ cannot be uniformized over all the points added to the domain. If $R$ has cofinite domain this cannot be done, so we will replace $R$ by $\baire\times R$, and then add noise to extend it to a full domain relation.

Let $N\subseteq \baire^3$ be such that $N_{(x,y)}\subseteq \baire\sm \Delta^1_1(x,y)$. Note that if $f$ is a $\Delta^1_1(p)$ function whose graph is contained in $N$, then $\operatorname{dom}(f)\cap \{p\}\times \baire=\emptyset$.   

Given $R\subseteq \baire^2$ closed, let $R'=\left(\baire\times R\right)\sqcup N$. If $R$ admits full Borel uniformization, say via $f$, then so does $R'$, via $f'(x,y)=f(y)$. If $R'$ admits Borel uniformization, say via $f\in\Delta^1_1(p)$, then for any $x$, \[f(p,x)\in R'_{(p,x)}\cap \Delta^1_1(p,x)\subseteq R'_{(p,x)}\sm N_{(p,x)}= R_x.\] So $R$ admits full Borel uniformization via $f'(x)=f(p,x).$

Identifying $\baire^2$ with $\baire$ via a Borel isomorphism as usual, the map $R\mapsto R'$ is a reduction from $\mathbf{FBU}$ to $\mathbf{Uni}$.

\paragraph{$\mathbf{Uni}\leq_B\mathbf{Sel}$:} If $R\subseteq\baire^2$
 is closed, define 
 \[(x,y) E_R (x',y') :\lra x=x'\wedge \left[ (x,y), (x,y')\in R\vee y=y'\right].\] Then, $E_R$ is a closed equivalence relation.
 
 If $R$ admits Borel uniformization, say via $f$, then $E_R$ has a Borel selector $g$ defined by
 \[g(x,y)=\left\{\begin{array}{ll} (x,f(x)) & (x,y)\in R \\ (x,y) & else \end{array}\right.\]
 
 If $E_R$ admits a selector, $g$, then $R$ admits uniformization $f$ via $f(x)=y:\lra g(x,y)=(x,y)\wedge (x,y)\in R$.
 
 Again, $\baire^4$ and $\baire^2$ can be identified, and the map $R\mapsto E_R$ is Borel. So we have a reduction as claimed.
\end{proof}

Note that, by Harrington--Kechris--Louveau \cite{HKL}, the set of smooth equivalence relations is $\bp11$ in the codes. So the preceding theorem gives a strong reason why admitting a selector is not the same as being smooth.

\end{document}